\documentclass[12pt]{amsart}

\usepackage{amsmath,amsthm,amsfonts,amssymb,latexsym,amscd,bbm}
\usepackage{color}
\usepackage{cite}
\begin{document}

\newtheorem{theorem}{Theorem}
\newtheorem{corollary}[theorem]{Corollary}
\newtheorem{lemma}[theorem]{Lemma}
\newtheorem{proposition}[theorem]{Proposition}
\newtheorem{conjecture}[theorem]{Conjecture}
\newtheorem{commento}[theorem]{Comment}
\newtheorem{problem}[theorem]{Problem}
\newtheorem{remarks}[theorem]{Remarks}
\newtheorem{example}[theorem]{Example}

\theoremstyle{definition}
\newtheorem{definition}[theorem]{Definition}

\theoremstyle{remark}
\newtheorem{remark}[theorem]{Remark}

\newcommand{\Nb}{{\mathbb{N}}}
\newcommand{\Rb}{{\mathbb{R}}}
\newcommand{\Tb}{{\mathbb{T}}}
\newcommand{\Zb}{{\mathbb{Z}}}
\newcommand{\Cb}{{\mathbb{C}}}

\newcommand{\Af}{\mathfrak A}
\newcommand{\Bf}{\mathfrak B}
\newcommand{\Ef}{\mathfrak E}
\newcommand{\Gf}{\mathfrak G}
\newcommand{\Hf}{\mathfrak H}
\newcommand{\Kf}{\mathfrak K}
\newcommand{\Lf}{\mathfrak L}
\newcommand{\Mf}{\mathfrak M}
\newcommand{\Rf}{\mathfrak R}

\newcommand{\x}{\mathfrak x}

\def\A{{\mathcal A}}
\def\B{{\mathcal B}}
\def\C{{\mathcal C}}
\def\D{{\mathcal D}}
\def\E{{\mathcal E}}
\def\F{{\mathcal F}}
\def\G{{\mathcal G}}
\def\H{{\mathcal H}}
\def\J{{\mathcal J}}
\def\K{{\mathcal K}}
\def\LL{{\mathcal L}}
\def\N{{\mathcal N}}
\def\M{{\mathcal M}}
\def\N{{\mathcal N}}
\def\OO{{\mathcal O}}
\def\P{{\mathcal P}}
\def\Q{{\mathcal Q}}
\def\SS{{\mathcal S}}
\def\T{{\mathcal T}}
\def\U{{\mathcal U}}
\def\W{{\mathcal W}}

\def\ext{\operatorname{Ext}}
\def\clsp{\overline{\operatorname{span}}}
\def\Ad{\operatorname{Ad}}
\def\ad{\operatorname{Ad}}
\def\tr{\operatorname{tr}}
\def\id{\operatorname{id}}
\def\en{\operatorname{End}}
\def\aut{\operatorname{Aut}}
\def\out{\operatorname{Out}}
\def\coker{\operatorname{coker}}

\def\la{\langle}
\def\ra{\rangle}
\def\rh{\rightharpoonup}
\def\cl{\textcolor{blue}{$\clubsuit$}}

\newcommand{\Aut}[1]{\text{Aut}(#1)}
\newcommand{\Fn}{\mathcal{F}_n}
\newcommand{\On}{\mathcal{O}_n}
\newcommand{\Dn}{\mathcal{D}_n}
\newcommand{\Sn}{\mathcal{S}_n}

\newcommand{\norm}[1]{\left\lVert#1\right\rVert}
\newcommand{\inpro}[2]{\left\langle#1,#2\right\rangle}

\title[On Conjugacy of Subalgebras of Graph $C^*$-Algebras]{On Conjugacy of Subalgebras of Graph \\ $C^*$-Algebras}


\medskip\noindent
\author[T. Hayashi]{Tomohiro Hayashi}\\
\address{Nagoya Institute of Technology\\
Gokiso-cho, Showa-k, Nagoya 466--8555, Japan}
\email{hayashi.tomohiro@nitech.ac.jp}

\author[J.H. Hong]{Jeong Hee Hong}
\address{Department of Data Information\\
  Korea Maritime and Ocean University\\
  Busan 49112, South Korea}
\email{hongjh@kmou.ac.kr}

\author[S.E. Mikkelsen]{\mbox{Sophie Emma Mikkelsen}}
\address{Department of Mathematics and Computer Science\\
  The University of Southern Denmark\\
  Campusvej 55, DK--5230 Odense M, Denmark}
\email{mikkelsen@imada.sdu.dk}

\smallskip\noindent
\author[W. Szyma\'{n}ski]{Wojciech Szyma{\'n}ski}
\address{Department of Mathematics and Computer Science\\
The University of Southern Denmark\\
Campusvej 55, DK--5230 Odense M, Denmark}
\email{szymanski@imada.sdu.dk}



\thanks{J. H. Hong was supported by Basic Science 
Research Program through the National Researc Foundation of Korea (NRF) funded by te Ministry of Education, 
Science and Technology (Grant No. 2016R1D1A1B03930839).  
S. E. Mikkelsen and W. Szyma\'{n}ski were supported by  the DFF-Reesearch Project 2 on `Automorphisms and invariants of 
operator algebras', Nr. 7014--00145B.}

\date{January 31, 2019}

\begin{abstract}The problem of inner vs outer conjugacy of subalgebras of certain graph $C^*$-algebras is investigated. 
For a large class of finite graphs $E$, we show that whenever $\alpha$ is a vertex-fixing quasi-free automorphism of the 
corresponding graph $C^*$-algebra $C^*(E)$ such that $\alpha(\D_E)\neq\D_E$, where $\D_E$ is the canonical 
MASA in $C^*(E)$, then $\alpha(\D_E)\neq w\D_E w^*$ for all unitaries $w\in C^*(E)$. That is, the two MASAs 
$\D_E$ and $\alpha(\D_E)$ of $C^*(E)$ are outer but not inner conjugate. For the 
Cuntz algebras $\OO_n$, we find a criterion which guarantees that a polynomial automorphism moves the canonical UHF subalgebra 
to a non-inner conjugate UHF subalgebra. The criterion is phrased in terms of rescaling of trace on diagonal projections. 
\end{abstract}

\maketitle



\section{Introduction}

Maximal abelian subalgebras (MASAs) have played very important role in the study of von Neumann algebras 
from the very beginning, and their theory is quite well developed by now. Theory of MASAs of 
$C^*$-algebras is somewhat less advanced, several nice attempts in this direction notwithstanding. Our particular 
interest lies in classification of MASAs in purely infinite simple $C^*$-algebras, and especially in Kirchberg algebras. 
In addition to its intrinsic interest, better understanding of MASAs in Kirchberg algebras could have significant consequences 
for the  classification of automorphisms and group actions on these algebras. In this context, we would like to 
single out the recent work of Barlak and Li, \cite{BaLi}, where a connection between the outstanding UCT problem for 
crossed products and existence of invariant Cartan subalgebras is investigated. 

It is a very difficult problem if two outer conjugate MASAs (that is, two MASAs $\Af$ and $\Bf$ for which there 
exists an automorphism $\sigma$ of the ambient algebra such that $\sigma(\Af)=\Bf$) of a purely infinite simple 
$C^*$-algebra are inner conjugate as well (that is, if there exists a unitary $w$ such that $w\Af w^*=\Bf$). This question was 
answered to the negative in \cite[Theorem 3.7]{CHS2015} for quasi-free automorphisms of the Cuntz algebras $\OO_n$. 

In the present paper, we extend the main result of \cite{CHS2015} to the case of purely infinite simple graph 
$C^*$-algebras $C^*(E)$ corresponding to finite graphs $E$. Namely, we show in Theorem \ref{mainqfree} below that 
every quasi-free automorphism of $C^*(E)$ either leaves the canonical MASA $\D_E$ globally invariant or moves it to 
another MASA of $C^*(E)$ which is not inner conjugate to $\D_E$. Although our Theorem \ref{mainqfree} is stated 
for quasi-free automorphisms only, it is in fact applicable to some other automorphisms as well. This is due to the fact 
that passing from one graph $E$ to another $F$ with the isomorphic algebra $C^*(F)\cong C^*(E)$ will often not preserve 
the property of an automorphism to be quasi-free. 
To make the present paper self-contained, we recall the necessary background on graph 
$C^*$-algebras and their endomorphisms in the preliminaries. 

The problem of conjugacy of subalgebras has been mostly investigated in the context of MASAs. However, it is very interesting 
for other types of subalgebras as well. In the present paper, we initiate systematic investigations of the outer vs inner conjugacy for the 
canonical UHF-subalgebra $\F_n$ of the Cuntz algebra $\OO_n$. More specifically, we address the question if $\F_n$ may be inner conjugate to 
$\lambda_u(\F_n)$, where $\lambda_u$ is a polynomial automorphism of $\OO_n$, building on the  first observations in this direction made 
in \cite{CHSam}.  Our results have clear potential for shedding more light on the mysterious structure of the outer automorphism group of 
$\OO_n$.  


\section{Preliminaries}

\subsection{Finite directed graphs and their $C^*$-algebras}

Let $E=(E^0,E^1,r,s)$ be a directed graph, where $E^0$ and $E^1$ are {\em finite} sets of vertices 
and edges, respectively, and $r,s:E^1\to E^0$ are range and source maps, respectively. 
A {\em path} $\mu$ of length $|\mu|=k\geq 1$ is a sequence 
$\mu=(\mu_1,\ldots,\mu_k)$ of $k$ edges $\mu_j$ such that 
$r(\mu_j)=s(\mu_{j+1})$ for $j=1,\ldots, k-1$. We view the vertices as paths of length $0$.
The set of all paths of length $k$ is denoted $E^k$, and $E^*$ denotes the collection of 
all finite paths (including paths of length zero). The range and source maps naturally 
extend from edges $E^1$ to paths $E^k$. A {\em sink} is a vertex $v$ which emits 
no edges, i.e. $s^{-1}(v)=\emptyset$. By a {\em cycle} we mean a path $\mu$ of 
length $|\mu|\geq 1$ such that $s(\mu)=r(\mu)$. 
A cycle $\mu=(\mu_1,\ldots,\mu_k)$ has an exit if there is a $j$ such that $s(\mu_j)$ 
emits at least two distinct edges. Graph $E$ is {\em transitive} if for any two vertices $v,w$ there exists 
a path $\mu\in E^*$ from $v$ to $w$ of non-zero length. Thus a transitive graph does not contain any 
sinks or sources. Given a graph $E$, we will denote by $A=[A(v,w)]_{v,w\in E^0}$ its {\em adjacency matrix}.  
That is, $A$ is a matrix with rows and columns indexed by the vertices of $E$, such that $A(v,w)$ is 
the number of edges with source $v$ and range $w$. 

The $C^*$-algebra $C^*(E)$ corresponding to a graph $E$ 
is by definition, \cite{KPRR} and \cite{KPR},  the universal $C^*$-algebra generated by mutually 
orthogonal projections $P_v$, $v\in E^0$, and partial isometries $S_e$, $e\in E^1$, 
subject to the following two relations: 
\begin{itemize}
\item[]
\begin{itemize}
\item[(GA1)] $S_e^*S_e=P_{r(e)}$,  
\item[(GA2)] $P_v=\sum_{s(e)=v}S_e S_e^*$ if $v\in E^0$ emits at least one edge. 
\end{itemize}
\end{itemize}
For a path $\mu=(\mu_1,\ldots,\mu_k)$ we denote by $S_\mu=
S_{\mu_1}\cdots S_{\mu_k}$ the corresponding partial isometry in $C^*(E)$.  
 We agree to write $S_v=P_v$ for a  $v\in E^0$.
Each $S_\mu$ is non-zero with the domain projection $P_{r(\mu)}$. 
Then $C^*(E)$ is the closed span of $\{S_\mu S_\nu^*:\mu,\nu\in E^*\}$.  
Note that $S_\mu S_\nu^*$ is non-zero if and only if 
$r(\mu)=r(\nu)$. In that case, $S_\mu S_\nu^*$ is a partial isometry with domain and range projections 
equal to $S_\nu S_\nu^*$ and $S_\mu S_\mu^*$, respectively. 

The range projections $P_\mu=S_\mu S_\mu^*$ of all 
partial isometries $S_\mu$ mutually commute, and the abelian $C^*$-subalgebra of $C^*(E)$ 
generated by all of them is called the diagonal subalgebra and denoted $\D_E$. 
We set $\D^0_E = {\rm span}\{P_v  :  v\in E^0  \}$ and, more generally, 
$\D_E^k= {\rm span}\{P_\mu  :  \mu\in E^k  \}$ for $k\geq 0$. $C^*$-algebra $\D_E$ coincides with the closed linear span of $\bigcup_{k=0}^\infty\D_E^k$. 
If $E$ does not contain sinks and all cycles have exits then 
$\D_E$ is a MASA (maximal abelian subalgebra) in $C^*(E)$ by \cite[Theorem 5.2]{HPP}. 
Throughout this paper, we make the following

\vspace{2mm}\noindent
{\bf standing assumption:}  all graphs we consider are transitive  
and all cycles in these graphs admit exits. 

\vspace{2mm}
There exists a strongly continuous action $\gamma$ of the circle group $U(1)$ on $C^*(E)$, 
called the {\em gauge action}, such that $\gamma_z(S_e)=zS_e$ and $\gamma_z(P_v)=P_v$ 
for all $e\in E^1$, $v\in E^0$ and $z\in U(1)\subseteq\Cb$. 
The fixed-point algebra $C^*(E)^\gamma$ for the gauge action is an AF-algebra, denoted 
$\F_E$ and called the core AF-subalgebra of $C^*(E)$. $\F_E$ is the closed span of 
$\{S_\mu S_\nu^*:\mu,\nu\in E^*,\;|\mu|=|\nu|\}$. For $k\in\Nb=\{0,1,2,\ldots\}$ 
we denote by $\F_E^k$ the linear span of $\{S_\mu S_\nu^*:\mu,\nu\in E^*,\;|\mu|=|\nu|= k\}$. 
$C^*$-algebra $\F_E$ coincides with the norm closure of $\bigcup_{k=0}^\infty\F_E^k$. 

We consider the usual {\em shift} on $C^*(E)$, \cite{CK}, given by
\begin{equation}\label{shift}
\varphi(x)=\sum_{e\in E^1} S_e x S_e^*, \;\;\; x\in C^*(E).  
\end{equation} 
In general, for finite graphs without sinks and sources, the shift is a unital, completely positive map. However, it 
is an injective $*$-homomorphism when restricted to the relative commutant $(\D_E^0)'\cap C^*(E)$. 

We observe that for each $v\in E^0$ projection $\varphi^k(P_v)$ is minimal in the center of $\F_E^k$. 
The $C^*$-algebra $\F_E^k\varphi^k(P_v)$ is the linear span of partial isometries $S_\mu S_\nu^*$ with 
$|\mu|=|\nu|=k$ and $r(\mu)=r(\nu)=v$. It is isomorphic to the full matrix algebra of size $\sum_{w\in E^0}
A^k(w,v)$. The multiplicity of $\F_E^k\varphi^k(P_v)$ in $\F_E^{k+1}\varphi^{k+1}(P_w)$ is $A(v,w)$, 
so the Bratteli diagram for $\F_E$ is induced from the graph $E$, see \cite{CK}, \cite{KPRR} or \cite{BPRSz}.                                                                  
We also note that the relative commutant 
of $\F_E^k$ in $\F_E^{k+1}$ is isomorphic to $\bigoplus_{v,w\in E^0}M_{A(v,w)}(\Cb)$. 

For an integer $m\in\Zb$, we denote by $C^*(E)^{(m)}$ the spectral subspace of the gauge 
action corresponding to $m$. That is,  
$$ C^*(E)^{(m)}:=\{x\in C^*(E) \mid \gamma_z(x)=z^m x,\,\forall z\in U(1)\}. $$ 
In particular, $C^*(E)^{(0)}=C^*(E)^\gamma$. 
There exist faithful conditional expectations $\Phi_{\F}:C^*(E)\to\F_E$ and $\Phi_{\D}:C^*(E)\to\D_E$ 
such that $\Phi_{\F}(S_\mu S_\nu^*)=0$ for $|\mu|\neq|\nu|$ and $\Phi_{\D}(S_\mu S_\nu^*)=0$ 
for $\mu\neq\nu$. Combining $\Phi_\F$ with a faithful conditional expectation from $\F_E$ onto $\F_E^k$, 
we obtain a faithful conditional expectation $\Phi_{\F}^k:C^*(E)\to\F_E^k$. Furthermore, for each 
$m\in\Zb$ there is a unital, contractive and completely bounded map $\Phi^m:C^*(E)\to C^*(E)^{(m)}$ 
given by 
\begin{equation}
\Phi^m(x) = \int_{z\in U(1)} z^{-m}\gamma_z(x)dz. 
\end{equation}
In particular, $\Phi^0=\Phi_\F$.
We have $\Phi^m(x)=x$ for all $x\in C^*(E)^{(m)}$.  If 
$x\in C^*(E)$ and $\Phi^m(x)=0$ for all $m\in\Zb$ then $x=0$.


\subsection{Endomorphisms determined by unitaries}

Cuntz's classical approach to the study of endomorphisms of $\OO_n$, \cite{Cun}, has recently been 
extended to graph $C^*$-algebras in \cite{CHS2015} and \cite{AJSz}. In this subsection, we recall a few 
most essential definitions and facts about such endomorphisms. 

We denote by $\U_E$ the collection of all those unitaries in $C^*(E)$ which 
commute with all vertex projections $P_v$, $v\in E^0$. That is 
\begin{equation}\label{ue}
\U_E:=\U((\D_E^0)'\cap C^*(E)). 
\end{equation}
If $u\in\U_E$ then $uS_e$, $e\in E^1$, are partial isometries in $C^*(E)$ which together with 
projections $P_v$, $v\in E^0$, satisfy (GA1) and (GA2). Thus, by the universality of $C^*(E)$, 
there exists a unital $*$-homomorphism $\lambda_u:C^*(E)\to C^*(E)$ such 
that\footnote{The reader should be aware that in some papers (e.g. in \cite{Cun}) a 
different convention is used, namely $\lambda_u(S_e)=u^* S_e$.}
\begin{equation}\label{lambda}
\lambda_u(S_e)=u S_e \;\;\; {\rm and}\;\;\;  \lambda_u(P_v)=P_v, \;\;\;  {\rm for}\;\; e\in E^1,\; v\in E^0. 
\end{equation}
The mapping $u\mapsto\lambda_u$ establishes 
a bijective correspondence between $\U_E$ and the semigroup of those unital endomomorphisms 
of $C^*(E)$ which fix all  $P_v$, $v\in E^0$.
As observed in \cite[Proposition 2.1]{CHS2012}, if $u\in\U_E\cap\F_E$ then $\lambda_u$ 
is automatically injective. We say $\lambda_u$ is {\em invertible} if $\lambda_u$ is an automorphism of $C^*(E)$. 
We denote 
\begin{equation}
 \Bf := (\D_E^0)'\cap \F_E^1. 
\end{equation}
That is, $\Bf$ is the linear span of elements $S_e S_f^*$, $e,f\in E^1$, with $s(e)=s(f)$ and $r(e)=r(f)$. We note that 
$\Bf$ is contained in the multiplicative domain of $\varphi$ and we have $\D_E^1 \subseteq \Bf 
\subseteq \F_E^1$. If $u\in\U(\Bf)$ then $\lambda_u$ is automatically invertible with inverse 
$\lambda_{u^*}$ and the map 
\begin{equation}\label{quasifree}
\U(\Bf) \ni u\mapsto \lambda_u \in\aut(C^*(E))
\end{equation} 
is a group homomorphism with range inside 
the subgroup of {\em quasi-free automorphisms} of $C^*(E)$, see \cite{Z}. Note that this group is almost never trivial 
and it is non-commutative if graph $E$ contains two edges $e,f\in E^1$ such that $s(e)=s(f)$ and $r(e)=r(f)$. 

The shift $\varphi$ globally preserves $\U_E$, $\F_E$ and $\D_E$. For $k\geq 1$ we denote 
\begin{equation}\label{uk}
u_k := u\varphi(u)\cdots\varphi^{k-1}(u).   
\end{equation}
For each $u\in\U_E$ and all $e\in E^1$ we have $S_e u = \varphi(u) S_e$, and thus 
\begin{equation}\label{uaction}
\lambda_u(S_\mu S_\nu^*)=u_{|\mu|}S_\mu S_\nu^*u_{|\nu|}^* 
\end{equation}
for any two paths $\mu,\nu\in E^*$. 


\section{Quasi-free automorphisms}

In this section, we extend the main result of \cite{CHS2015}, applicable to the Cuntz algebras, to a much wider 
class of graph $C^*$-algebras. 

For the proof of Lemma \ref{keylemma}, below, we recall from Lemma 3.2 and Remark 3.3 in \cite{CHS2015} 
that if $x\in C^*(E)$, $x\geq 0$, and $x\D_E=\D_E x$ then $x\in\D_E$.

\begin{lemma}\label{keylemma}
Let $u\in\U(\Bf)$ be such that $u\D_E^1u^*\neq\D_E^1$, and let $x\in\F_E$ be 
arbitrary. If $x\lambda_u(\D_E) = \D_Ex$ then $x=0$.  
\end{lemma}
\begin{proof}
Suppose $x\in\F_E$ is such that $||x||=1$ and $x\lambda_u(\D_E) = \D_Ex$.  From this we will derive a 
contradiction. 

Since $u\D_E^1u^*\neq\D_E^1$, there exists a vertex $v\in E^0$ such that $u\D_E^1 u^* P_v \neq \D_E^1 P_v$. 
Thus, since $u\D_E^1P_vu^*=   u\D_E^1 u^* P_v$, we can take a projection $p\in\D_E^1P_v$ satisfying
$\delta := \inf\{ ||upu^* -q|| \mid q\in\D_E^1\} >0. $ Since $\Phi_{\F}^1(q')\in\D_E^1$, for all $q'\in\D_E$ we get 
$$ ||upu^*-q'|| \geq ||\Phi_\F^1(upu^*-q')|| = || upu^* - \Phi_\F^1(q') || \geq \delta. $$
By assumption, for each $k\in\Nb$ there is a $q_k\in\D_E$ such that 
\begin{equation}\label{leq1}
x\lambda_u(\varphi^k(p))=q_kx. 
\end{equation} 
Since $u_k\in\F_E^k$ and $\varphi^k(upu^*)\in\varphi^k(\Bf)=(\F_E^k)'\cap\F_E^{k+1}$, we have
\begin{equation}\label{leq2}
\lambda_u(\varphi^k(p)) = u_k\varphi^k(\lambda_u(p))u_k^* = u_k\varphi^k(upu^*))u_k^* = \varphi^k(upu^*). 
\end{equation}
Identities (\ref{leq1}) and (\ref{leq2}) combined yield 
$ 0 = x\lambda_u(\varphi^k(p)) - q_kx = x\varphi^k(upu^*) - q_kx.$ Since $upu^*\in\Bf$, 
the sequence $\{\varphi^k(upu^*)\}_{k=1}^\infty$ is central in $\F_E$. 
Therefore we have $ \lim_{k\to\infty}(\varphi^k(upu^*)-q_k)xx^*=0. $ It follows from the assumption on $x$ that 
$xx^*\D_E = \D_E xx^*$, and thus we may conclude that  $xx^*\in\D_E$. 

Now, take an arbitrary $\epsilon >0$. For a sufficiently large $m\in\Nb$, we have 
$$ \limsup_{k\to\infty} ||(\varphi^k(upu^*)-q_k)\Phi_\F^m(xx^*)|| \leq \epsilon \;\;\; \text{and} 
\;\;\; ||\Phi_\F^m(xx^*)|| \geq 1-\epsilon. $$
Thus we can find a projection $d\in\D_E^m$ such that $ \limsup_{k\to\infty} ||(\varphi^k(upu^*)-q_k)d|| 
\leq \frac{\epsilon}{1-\epsilon}. $

Since graph $E$ is transitive, for a sufficiently large $k\in\Nb$ we can find a path $\mu\in E^k$ such that 
$r(\mu)=v$ and $S_\mu S_\mu^* \leq d$. But now we see that 
$$ 3\epsilon \geq ||(\varphi^k(upu^*)-q_k)d|| \geq ||(\varphi^k(upu^*)-q_k)S_\mu S_\mu^*|| = 
|| upu^*P_v - S_\mu^*q_kS_\mu|| \geq \delta. $$
Since $\epsilon$ can be arbitrarily small, this is the desired contradiction. 
\end{proof}

Now, we are ready to prove our main result. 

\begin{theorem}\label{mainqfree}
Let $u\in\U(\Bf)$ be such that $u\D_E^1u^*\neq\D_E^1$. Then there is no non-zero element 
$x\in C^*(E)$ satisfying $x\lambda_u(\D_E) = \D_Ex$. In particular, there is no unitary $w\in C^*(E)$ such 
that $w\D_Ew^*=\lambda_u(\D_E)$. 
\end{theorem}
\begin{proof}
Let $x\in C^*(E)$ be such that $x\lambda_u(\D_E) = \D_Ex$. To verify that $x=0$, it suffices to show 
that $\Phi^m(x)=0$ for all $m\in\Zb$. 

We have $S_\mu^*\D_E S_\mu = P_{r(\mu)}\D_E $ for each $\mu\in E^*$. Thus $P_{r(\mu)}x
\lambda_u(\D_E) = P_{r(\mu)}\D_E x = S_\mu^*\D_E S_\mu x$, and hence $S_\mu x\lambda_u(\D_E) = 
\D_E S_\mu x$. Therefore by Lemma \ref{keylemma}, we get 
$$ \Phi_\F(S_\mu x) = 0 \;\;\; \text{for all} \;\;\; \mu\in E^*. $$
Let $m\in\Nb$. For a vertex $v\in E^0$ take a path $\mu\in E^m$ with $r(\mu)=v$. Then 
$0 = \Phi_\F(S_\mu x) = S_\mu\Phi^{-m}(x)$. Thus $P_v\Phi^{-m}(x)=0$, and summing over all $v\in E^0$ 
we see that $\Phi^{-m}(x)=0$ for all $m\in\Nb$. 

Now, taking adjoints of both sides of the identity $x\lambda_u(\D_E) = \D_Ex$ and then applying 
$\lambda_{u^*}=\lambda_u^{-1}$, we get $\lambda_{u^*}(x^*)\lambda_{u^*}(\D_E) = 
\D_E\lambda_{u^*}(x^*)$. Since $u^*\in\U(\Bf)$ and $u^*\D_E^1u\neq\D_E^1$, 
applying the preceding argument, we get $\Phi^{-m}(\lambda_{u^*}(x^*))=0$ for all $m\in\Nb$. But 
$\Phi^{-m}(\lambda_{u^*}(x^*)) = \lambda_{u^*}(\Phi^{-m}(x^*)) = \lambda_{u^*}(\Phi^{m}(x))$. 
Thus $\Phi^m(x)=0$ for all $m\in\Nb$, and the proof is complete. 
\end{proof}

\begin{corollary}
Let $u,v\in\U(\Bf)$ be such that $u\D_E^1u^*\neq v\D_E^1 v^*$. Then there is no unitary $w\in C^*(E)$ such 
that $w\lambda_u(\D_E)w^*=\lambda_v(\D_E)$. 
\end{corollary}


\section{Conjugacy by polynomial automorphisms of the UHF-subalgebra of the Cuntz algebra}

In this section we give a condition for inner conjugacy by polynomial automorphism of the core UHF-subalgebra
of the Cuntz algebra $\OO_n$, using the unique normalized trace on $\Fn$ which will be denoted by $\tau$. 

Let $W_n^k$ be the set of tuples $\mu=(\mu_1,...,\mu_k)$ where $\mu_i\in \{1,...,n\}$, and define 
$W_n=\bigcup_{k=0}^{\infty} W_n^k$ where $W_n^0=\{0\}$. We denote by $\Sn$ the group of unitaries in 
$\On$ which can be written as finite sums of words. Hence an element $u\in \Sn$ is of the form 
$u=\sum_{(\alpha,\beta)\in J} S_{\alpha}S_{\beta}^*$, where $J$ is a finite collection of pairs $(\alpha,\beta)$ 
with $\alpha,\beta\in W_n$. In the following we denote by $P(\Dn)$ the set of projections in $\Dn$. 

\begin{theorem}\label{UHFalgebra}
Let $u\in \Sn$ with $\lambda_u\in$ Aut$(\On)$. If there exists a sequence $\{P_n\}$ in $P(\Dn)$ such that  
\begin{equation}\label{trace}
\frac{\tau(\lambda_u(P_n))}{\tau(P_n)}\to \infty \ \text{or} \ 0,
\end{equation}
then for all $w\in\mathcal{U}(\On)$ and $v\in \mathcal{U}(\Fn)$ we have $\lambda_u\neq\Ad{w}\lambda_v$. 
This implies in particular that $\Fn$ and $\lambda_u(\Fn)$ are not inner conjugate. 
\end{theorem}

For the proof of Theorem \ref{UHFalgebra} we need the following result. 

\begin{lemma}\label{lemmaUHF}
If $u\in \mathcal{U}(\On)$ such that $u\Dn u^*\subseteq \Fn$ then $u$ has a finite Fourier series.  
\end{lemma}
\begin{proof}
Let $u$ have the Fourier series $ \sum_{k\in\mathbb{Z}} u_k, \ u_k:=\Phi^k(u).$
Let $P\in \Dn$ and fix $k\in \mathbb{Z}$ then 
$$
\begin{aligned}
uPu^*u_k = \int_{U(1)} z^{-k} \gamma_z(uPu^*u) dz  = \int_{U(1)} z^{-k} \gamma_z(uP) dt = u_kP .
\end{aligned}
$$
Where we have used that $u\Dn u^*\subseteq \Fn$ to write $uPu^* \gamma_z(u) =\gamma_z(uPu^*u)$.  

It follows that $P(u^*u_k)=(u^*u_k)P$ for $P\in \Dn$. Hence $u^*u_k\in \Dn' \cap \On$ and since $\Dn$ is a MASA there exists a $d_k\in\Dn$ 
such that $u_k=ud_k$. Now we have 
$$
u_k= \int_{U(1)} z^{-k} \gamma_z(u_k) dz = \int_{U(1)} z^{-k} \gamma_z(ud_k) dz 
= u_k d_k. 
$$
Hence $ud_k=ud_kd_k$ and $d_k=d_k^2$, which shows that $d_k$ is a projection in $\Dn$. 
For $k,m\in\mathbb{Z}$ with $k\neq m$ we have 
$$
0= \int_{U(1)} \gamma_z(u_k^*u_m)dz = \int_{U(1)} \gamma_z(d_k^*d_m)dz = d_kd_m.
$$
Hence $\{d_k\}$ are mutually orthogonal projections in $\Dn$. 

There exists an element 
of the form $w = \sum_j t_j S_{\alpha_j}S_{\beta_j}^*$ (finite sum) $t_j\in \mathbb{C}$,
such that $ || u-w || < \epsilon $. Also, there exists $N\in\mathbb{Z}$ such that for $k>N$ 
we have $\Phi^k(w)=0$, since $w$ has a finite Fourier series. Then
$$ \norm{u_k} = || \Phi^k(u) || = || \Phi^k(u-w) || \leq || u-w || < \epsilon, \   \text{for}\; k>N. $$ 
Hence $\norm{u_k}\to 0$ and thus
$\norm{d_k}\to 0$. Since $d_k$ are all projections the series $\{d_k\}$ is finite, concluding that $u$ has 
a finite Fourier series since $u_k=ud_k$. 
\end{proof}

\begin{proof}[Proof of Theorem \ref{UHFalgebra}]
We recall that 
if $\alpha\in$ Aut$(\On)$ then $\Fn$ and $\alpha(\Fn)$ are inner conjugate if and only if there exists 
$w\in \mathcal{U}(\On)$ such that $\alpha(\Fn)=\Ad{w}(\Fn)$. This is equivalent to that Ad$(w^*)\alpha(\Fn)=\Fn$. 
If we have an automorphism which globally preserves $\Fn$ then Ad$(w^*)\alpha$ is equal $\lambda_v$ for some 
$v\in \mathcal{U}(\Fn)$ \cite[Proposition 1.2(b)]{Cun},\cite[Proposition 3.3]{CS} therefore $\alpha=\Ad{w}\lambda_v$. 
Hence $\Fn$ and $\alpha(F_n)$ are inner conjugate if and only if there exist $w\in \mathcal{U}(\On)$ 
and $v\in \mathcal{U}(\Fn)$ s.t. $\alpha=\Ad w \lambda_v$. 

Assume that $\lambda_u=\Ad{w}\lambda_v$ for $w\in \mathcal{U}(\On)$ and $v\in \mathcal{U}(\Fn)$. 
We will assume $w\notin \Fn$ otherwise $\Ad{w}$ is trace preserving which contradicts  (\ref{trace}) since $\lambda_v$ 
is also trace preserving. We wish to proof that there are no sequence of projection such that we get (\ref{trace}). 

Since $\lambda_u\in\Aut{\On}$ we have $\lambda_v\in\Aut{\On}$ and $\Ad{w^*}=\lambda_v\lambda_u^{-1}$. 
Therefore $\Ad{w^*}(\Dn)\subseteq\Fn$ \cite[Proposition 3.3]{CS}. We wish to show that there exists $M>0$ such that 
\begin{equation}\label{Mequality}
\frac{\tau(w^*Pw)}{\tau(P)}\leq M, \  \text{for} \; P\in P(\Dn). 
\end{equation}
By Lemma \ref{lemmaUHF}, $w^*$ has a finite Fourier series and we can write 
$w^*=\sum_{j=1}^k {S_v^*}^jx_{-j}+x_0+\sum_{j=1}^k x_jS_v^j$, with $v\in \{1,2,...,n\}$ and $x_0,x_{\pm j}\in \Fn$. 

For any word $\mu$ with $|\mu|>k$ and $1\leq j\leq k$, we can write $\mu=\nu_j\mu_j$ such that 
$|\nu_j|=j$. Then $\tau(P_\mu)=\dfrac{1}{n^{|\mu|}}=
\dfrac{1}{n^j}\tau(P_{\mu_j})\geq \dfrac{1}{n^k}\tau(P_{\mu_j}).$ 

The only parts contributing to $\tau(w^*Pw)$ are
$$
\sum_{j=1}^k {S_v^*}^jx_{-j}Px_{-j}^*S_v^j, \ \ \sum_{j=1}^k x_jS_v^j P{S_v^*}^jx_j^*, \ \ x_0Px_0^*.
$$
We have ${S_v^*}^j x_{-j}S_{\nu_j}\in \Fn$ because $|v_j|=j$. Hence  
$$
\tau( {S_v^*}^j x_{-j}P_\mu x_{-j}^* S_v^j)
=\tau( (S_{\nu_j}^* x_{-j}^*S_v^j{S_v^*}^j x_{-j}S_{\nu_j})P_{\mu_j}).
$$ 
Note that ${S_v^j}^*x_{-j}$ is contractive since $\Phi^{-j}(w^*)={S_v^j}^*x_{-j}$ and $\Phi^j$ is contractive. 
By the Cauchy-Schwarz inequality we get $\tau( {S_v^*}^j x_{-j}P_\mu x_{-j}^* S_v^j)\leq \tau(P_{\mu_j})
\leq n^k \tau(P_\mu).$ On the other hand, using that $\Phi^j(w^*)=x_jS_v^j$ is contractive and the 
Cauchy-Schwarz inequality we have
$$
\begin{aligned}
\tau(x_jS_v^jP_\mu {S_v^*}^jx_j^*)&=
\tau((x_jS_v^j{S_v^j}^*)(S_v^jP_\mu {S_v^*}^j)(S_v^j{S_v^j}^*x_j^*)) 
\\
&\leq \tau(S_v^jP_\mu {S_v^*}^j)\leq \tau(P_\mu). 
\end{aligned} 
$$
We also have $\tau(x_0^*P_{\mu}x_0)=\tau(x_0x_0^*P_{\mu})\leq \tau(P_{\mu})$ since $\Phi^0(w^*)=x_0$ 
which is contractive. Hence there exists a constant $M>0$ satisfying $\tau(w^*P_\mu w)\leq M\tau(P_\mu) $ 
for any word $\mu$ with $|\mu|>k$. If $|\mu|\leq k$ we can extend the length until it is greater than $k$ using that 
$\sum_{i=1}^n S_iS_i^*=1$. Hence $\tau(w^*P w)\leq M\tau(P)$ for any $P\in P(\Dn)$. Since $\lambda_v$ is 
trace preserving we have
\begin{equation}\label{tracefraction}
\frac{\tau(\Ad{w^*}\lambda_u(P))}{\tau(P)}= \frac{\tau(\lambda_v(P))}{\tau(P)}=1, \ \text{for}\; P\in P(\Dn).  
\end{equation}
Hence $\frac{1}{M}\leq \frac{\tau(\lambda_u(P))}{\tau(P)}$ for $P\in P(\Dn)$ since 
$\frac{\tau(P)}{\tau(\lambda_u(P))}=\frac{\tau(\text{Ad}w^*(\lambda_u(P))}{\tau(\lambda_u(P))}\leq M$ for $P\in P(\Dn). $

Then there are no sequence of projections $\{P_n\}$ such that $\frac{\tau(\lambda_u(P_n))}{\tau(P_n)}$ tends to $0$. 
To show that it has no sequence of projections such that the limit is infinity we use that $\lambda_{u}(\Fn)$ is inner conjugate 
to $\Fn$ if and only if $\lambda_u^{-1}(\Fn)$ is inner conjugate to $\Fn$. Indeed, if $\lambda_u(\Fn)=w\Fn w^*$ then 
$\Fn=\lambda_u^{-1}(w)\lambda_u^{-1}(\Fn)\lambda_u^{-1}(w^*)$. 

Now, assume that there exists a sequence $\{P_n\}$ such that 
$\frac{\tau(\lambda_u(P_n))}{\tau(P_n)}\to \infty$, then 
$\frac{\tau(Q_n)}{\tau(\lambda_u^{-1}(Q_n))}\to \infty$ where $Q_n=\lambda_u(P_n)$. But then $\{Q_n\}$ is 
a sequence of projections such that  $\frac{\tau(\lambda_u^{-1}(Q_n))}{\tau(Q_n)}\to 0$. Since 
for $u\in \Sn$ we have $\lambda_u^{-1}=\lambda_v$ for some $v\in\Sn$, \cite[Theorem 2.1]{CS2}, we can use the same argument 
as above to show that there exists $M'>0$ such that $\frac{1}{M'}\leq \frac{\tau(\lambda_u^{-1}(P))}{\tau(P)}$ for $P\in P(\Dn)$. 
The claim follows.
\end{proof}

\begin{example}{\rm  Let $w=S_{22}S_{212}^*+S_{212}S_{22}^*+P_{211}+P_1\in \mathcal{S}_2$ then  
$$
\lambda_w(S_1)=S_1, \ \ \lambda_w(S_2)=S_2(S_2S_{12}^*+S_{12}S_2^*+P_{11}).
$$
Let $u=S_2S_{12}^*+S_{12}S_2^*+P_{11}\in \mathcal{S}_2$, note that $u\in\mathcal{U}(C^*(S_1))$. Let $\alpha_u$ be defined by $\alpha_u(S_1)=S_1$ and $\alpha_u(S_2)=S_2u$, then $\alpha_u$ is an automorphism with inverse $\alpha_{u^*}$ and $\lambda_w=\alpha_u$. Hence $\lambda_w$ is an automorphism of $\mathcal{O}_2$. Consider $\beta_k=(22...2)$, which is a word of length $k$ only containing $2'$s and let $\gamma_k=(21212...12)$ be a word of length $2k-1$ then 
$$
\lambda_w(P_{\beta_k})= S_2uS_2uS_2u\cdots S_2uu^*S_2^*u^*S_2^*\cdots u^*S_2^*= 
P_{\gamma_k}
$$
since $uS_2=S_{12}$. 
We then have 
$\frac{\tau(\lambda_w(P_{\beta_k}))}{\tau(P_{\lambda_k})}=\frac{1}{2^{k-1}}$
which tends to $0$ as $k\to \infty$. Hence $\mathcal{F}_2$ and $\lambda_w(\mathcal{F}_2)$ are not inner conjugate by Theorem \ref{UHFalgebra} using the sequence $\{P_{\beta_k}\}$. }
\end{example}


\end{document}